\documentclass[letterpaper, 10 pt, conference]{ieeeconf}  %

\IEEEoverridecommandlockouts                              %
\usepackage{tikz}
\usepackage{subcaption}
\usepackage{booktabs}
\usepackage{url}
\usepackage{cite}
\usepackage{amsmath,amssymb,amsfonts}
\usepackage{algorithmic}
\usepackage{graphicx}
\usepackage{textcomp}
\usepackage{xcolor}
\def\BibTeX{{\rm B\kern-.05em{\sc i\kern-.025em b}\kern-.08em
    T\kern-.1667em\lower.7ex\hbox{E}\kern-.125emX}}
\newtheorem{thm}{Theorem }%
\newtheorem{proposition}{Proposition}%
\newtheorem{lemma}{Lemma}%
\newtheorem{defn}{Definition}%
\newtheorem{rem}{Remark }%
\newtheorem{assumption}{Assumption}%
\newtheorem{example}{Example}[section]

\usepackage[ruled,vlined,linesnumbered]{algorithm2e}
\usepackage{makecell}

\usepackage{enumerate}

\newcommand{\set}[1]{\left\{#1\right\}}
\newcommand{\norm}[1]{\left\Vert #1 \right\Vert}
\newcommand{\abs}[1]{\left\vert #1 \right\vert}
\newcommand{\ra}{\rightarrow}
\newcommand{\Real}{\mathbb{R}}
\newcommand{\eps}{\varepsilon}

\renewcommand{\subset}{\subseteq}

\newcommand{\K}{\mathcal{K}}
\newcommand{\C}{\mathcal{C}}
\newcommand{\D}{\mathcal{D}}

\usepackage{xcolor}

\title{\LARGE \bf
A Converse Control Lyapunov Theorem for Joint Safety and Stability
}

\author{Thanin Quartz, Maxwell Fitzsimmons, and Jun Liu  %
\thanks{This research was supported in part by an NSERC Discover Grant and the Canada Research Chairs program.}%
\thanks{Thanin Quartz, Maxwell Fitzsimmons, and Jun Liu are with the Department of Applied Mathematics, University of Waterloo, Waterloo, Ontario N2L 3G1, Canada.  Email: \texttt{j.liu@uwaterloo.ca (Jun Liu)}
        }%
}

\begin{document}

\maketitle
\thispagestyle{empty}
\pagestyle{empty}

\begin{abstract}
We show that the existence of a strictly compatible pair of control Lyapunov and control barrier functions is equivalent to the existence of a single continuously differentiable Lyapunov function that certifies both asymptotic stability and safety. This characterization complements existing literature on converse Lyapunov functions by establishing a partial differential equation (PDE) characterization with prescribed boundary conditions on the safe set, ensuring that the safe set is exactly certified by this Lyapunov function. The result also implies that if a safety and stability specification cannot be certified by a single Lyapunov function, then any pair of control Lyapunov and control barrier functions \textit{necessarily} leads to a conflict and cannot be satisfied simultaneously in a robust sense. 
\end{abstract}

\section{Introduction}

As control applications become more complex, design challenges have gone beyond stability guarantees to include formal safety guarantees. When designing a feedback controller, it is essential both to stabilize the system and to enforce state constraints. This is true for applications in safety‐critical domains such as autonomous vehicles, chemical plants, and robotic systems that are only deployed after rigorous demonstrations of safety \cite{koopman2016challenges,crowl1990chemical}. 

Control Lyapunov functions (CLFs) are a well‐studied tool for designing stabilizing controllers. In particular, the necessary and sufficient conditions for the existence of a CLF were established in \cite{artstein1983stabilization}, and these results were subsequently used to design a universal control law for affine nonlinear systems \cite{sontag1989universal}. Similar to the CLF approach, control barrier functions (CBFs) \cite{wieland2007constructive,ames2016control} enforce safety by requiring the satisfaction of inequalities through the Lie derivative of a barrier function. However, while CBFs ensure safety, they do not, on their own, guarantee system stability. Consequently, safe stabilization problems, and related reach-avoid formulations, that integrate both safety and stability (or reachability) have received growing attention; see, e.g.,  \cite{romdlony2016stabilization,meng2022smooth,mestres2025conversetheoremscertificatessafety,li2024stabilization,dawson2022safe,li2023graphical,ong2019universal,mestres2022optimization,dai2024verification} for a partial list of results.

One common approach is to combine CBFs and CLFs in an optimization-based framework. When the system is control-affine, optimization can be achieved via quadratic programming \cite{ames2016control,ames2019control}, which is efficiently solvable and amenable to online implementation, such as in model predictive control (MPC) \cite{wu2019control}. However, this introduces new challenges, as the design must ensure compatibility between the two functions for the optimization problem to be feasible. When compatibility fails, the stability requirement is often relaxed and treated as a soft constraint \cite{ames2016control,li2023graphical,mestres2022optimization}. 

An alternative approach, taken in \cite{romdlony2016stabilization}, is to unify a CLF and a CBF into a single control Lyapunov–barrier function (CLBF). Once a CLBF has been obtained, Sontag’s universal formula can be applied to design a safe stabilizing controller. While practically appealing, a drawback of the approach in \cite{romdlony2016stabilization}, as later pointed out in \cite{braun2017existence,braun2020comment} (see also discussions in  \cite{meng2022smooth,mestres2025conversetheoremscertificatessafety}), is that the CLBF conditions stated in \cite{romdlony2016stabilization} cannot hold without stringent requirements on the safe set.

Recent work has sought to address these challenges by investigating converse Lyapunov theorems for joint safety and stability. In \cite{meng2022smooth}, the authors present a unified theoretical framework for Lyapunov and barrier functions via converse theorems that guarantee stability with safety, as well as reach-avoid-stay specifications. In a follow-up study, \cite{mestres2025conversetheoremscertificatessafety} provided a broader treatment of converse results by deriving necessary conditions for the existence of CLBFs and compatible CBF-CLF pairs. 

In this paper, we characterize CLBFs via strictly compatible CBF-CLF pairs. Specifically, 
complementary to the results in \cite{meng2022smooth,mestres2025conversetheoremscertificatessafety}, we provide a succinct proof that a strictly compatible pair of control Lyapunov and control barrier functions exists if and only if there is a single continuously differentiable Lyapunov function that simultaneously certifies asymptotic stability and safety. To the best of the authors’ knowledge, this is the first result showing that the existence of a compatible CLF–CBF pair implies the existence of a single CLBF, and this is achieved under a weaker notion of compatibility (see Remarks \ref{rem:boundary} and \ref{rem:controller}) Our results extend the converse Lyapunov literature by providing a PDE characterization with prescribed boundary conditions on the safe set, ensuring exact certification of safety. Furthermore, our result implies that if a safety and stability specification cannot be certified by a single continuously differentiable Lyapunov function, then any pair of control Lyapunov and control barrier functions \textit{necessarily} leads to a conflict and cannot be satisfied simultaneously in a robust sense.

\section{Preliminaries}

Consider a control-affine system of the form
\begin{equation}
    \label{eq:sys}
    \dot x = 
    f(x) + g(x)u,
\end{equation}
where $f:\,\Real^n\ra\Real^n$ and $g:\,\Real^n\ra\Real^{n\times m}$. We seek to design a feedback controller of the form $u=u(x)$ such that solutions of the closed-loop system
\begin{equation}
    \label{eq:clsys}
    \dot x = F(x) := f(x) + g(x)u(x)
\end{equation}
have desired properties. We assume that $f(0)=0$ and $u(0)=0$, i.e., the origin is an equilibrium point of (\ref{eq:clsys}). We assume that both $f$ and $g$ are continuously differentiable (except possible at the origin). The controller we aim to construct shall have the same property. Under these assumptions, we denote the unique solution to~(\ref{eq:clsys}) from the initial condition $x(0) = x \neq 0$ by $\phi(t, x)$. 

A set $\mathcal{C} \subset \mathbb{R}^n$ is said to be forward invariant for (\ref{eq:clsys}), if for all $x\in \mathcal{C}$, the solution $\phi(t,x) \in \mathcal{C}$ for all $t \geq 0$.
CBF methods in safe control rely on set invariance. A safe set $\mathcal{C}$ 
is completely characterized by the zero-superlevel set of a continuously differentiable function $h: \mathbb{R}^n \rightarrow \mathbb{R}$ as
\begin{align} 
\mathcal{C} & = \left\{ x \in \mathbb{R}^n \mid h(x) \geq 0 \right\}, \label{safe-set}\\
\partial \mathcal{C} & = \left\{ x \in \mathbb{R}^n \mid h(x) = 0 \right\}.
\end{align}

\begin{defn}[Extended class $\K$ function] We say that a continuous function $\alpha:[-b, a) \ra\Real$ $a,\,b>0$, belongs to extended class $\mathcal{K}$, if it is strictly increasing and $\alpha(0)=0$.
\end{defn}

CBFs \cite{xu2015robustness,ames2019control,xu2018constrained,ames2016control} are a constructive approach to verifying controlled forward invariance of a safe set.

\begin{defn}[Control barrier function]\label{defn:CBF}
    Given the set $\mathcal{C}$ as defined in (\ref{safe-set}), the continuously differentiable function $h$ is called a control barrier function for (\ref{eq:sys}) on a domain $\mathcal{D}$ with $\mathcal{C} \subset \mathcal{D} \subset \mathbb{R}^n$, if there exists an extended class $\K$ function $\alpha$ such that, for all $x\in \D$, 
\begin{equation} \label{eq:CBF-condition}
    \sup _{u \in \mathbb{R}^{m}}\left[L_f h(x)+L_g h(x) u\right]+\alpha(h(x)) \geq 0
\end{equation}
where $L_f h(x) = \nabla h(x)^{\top} f(x)$ and $L_g h(x) = \nabla h(x)^{\top} g(x)$.
\end{defn}

\begin{rem}\label{rem:boundary}
By convention \cite{ames2016control}, inequality~\eqref{eq:CBF-condition} is stated for all $x \in \D$. Theoretically, however, forward invariance of $\mathcal{C}$ follows from the same condition stated only on $\partial \mathcal{C}$, so it suffices to enforce it on the boundary. Moreover, for $x \in \partial \mathcal{C}$, we have $\alpha(h(x)) \equiv 0$. We use this simplified version later in our assumption (cf.~\eqref{eq:CBF} in Assumption~\ref{Assmp:SC}).
\end{rem}

It is well known that the existence of a CLF characterizes asymptotic stabilizability of system~(\ref{eq:sys})~\cite{artstein1983stabilization,sontag1983lyapunov,sontag1989universal}. In particular, a continuously differentiable (i.e., $C^1$) and radially unbounded CLF on $\Real^n$ implies global asymptotic stabilizability by a feedback law continuous on $\Real^n\setminus\set{0}$~\cite{artstein1983stabilization,sontag1983lyapunov}.

\begin{defn}[Control Lyapunov function]\label{defn:CLF}
    Given an open set $\D \subseteq \mathbb{R}^n$ containing the origin, a continuously differentiable function $V:\,\D \rightarrow \mathbb{R}$ is called a control Lyapunov function on $\D$ for the system (\ref{eq:sys}), if $V$ is positive definite on $\D$
    and, for each $x \in \D \backslash\{0\}$, 
\begin{equation}
    \inf_{u\in\Real^m} \left[L_f V(x) + L_g V(x) u \right] < 0.
\end{equation}
\end{defn}

There has been considerable effort in the literature to unify CLF and CBF approaches. At the intersection of stability and safety verification lies the concept of a control Lyapunov-barrier function (CLBF) \cite{romdlony2016stabilization}. In this paper, we provide a definition of a CLBF that is essentially a Lyapunov function (see, e.g., \cite{meng2022smooth,dawson2022safe}) with a prescribed boundary condition, exactly representing the safe set as its sub-level set. 

\begin{defn}[Control Lyapunov-barrier Function] \label{defn:CLBF}
Let $\D \subseteq \mathbb{R}^n$ be an open set containing the origin. A continuously differentiable function $W:\,\D\ra\Real$ is called a CLBF w.r.t. $\mathcal{C}$, if $\C\subset\D$ and $W$ is positive definite on $\D$ and satisfies the following conditions:
    \begin{enumerate}
        \item $\inf _{u \in \mathbb{R}^m} \nabla W(x)^{\top} \left[f(x)+g(x)u\right]<0,\, \forall x \in \mathcal{D}\setminus\set{0}$; 
        \item $\{ x\in \D : W(x) \le 1 \} = \mathcal{C}$; 
        \item $\{x\in \D : W(x) = 1 \} = \partial \mathcal{C}$.
    \end{enumerate}
\end{defn}

\begin{rem}
While the value $1$ is chosen for convenience, enforcing $W$ to be constant on $\partial\C$ ensures that $\C$ is exactly a sublevel set and $\partial\C$ a level set of $W$, guaranteeing forward invariance, and hence safety, of $\C$ under the resulting controller. Note that Definition~\ref{defn:CLBF} implies Definition~\ref{defn:CLF}, and the boundary of $\C$ is robustly controlled invariant since $W(x)$ is strictly decreasing along level sets under a suitable control input.  Therefore, Sontag's universal formula~\cite{sontag1989universal}, i.e., 
$u(x)=0$ when $L_g W(x)=0$ and 
$u(x)= -\dfrac{L_f W(x) + \sqrt{(L_f W(x))^2 + \|L_g W(x)\|^4}}{\|L_g W(x)\|^2}\, L_g W(x)$ otherwise, 
applies directly to $W(x)$ to obtain a safe, stabilizing controller. This controller is continuous on $\mathcal{C}\setminus\{0\}$, including points where $L_g W=0$, by an implicit function theorem argument and, if the small control property holds additionally, it is also continuous at $x=0$ \cite[Theorem 1]{sontag1989universal}.
\end{rem}

Definition \ref{defn:CLBF} provides a conceptually simpler characterization of joint safety and stability via a single control Lyapunov function. In the next section, we show, using a constructive approach, that this simplification is without loss of generality under the assumption of strictly compatible CBF-CLF pairs.

\section{Main results}

In this section, we establish the existence of a single $C^1$ Lyapunov function that certifies both asymptotic stability and safety, given the existence of a strictly compatible pair of a CLF and a CBF. 
The main result follows in two steps: Proposition \ref{prop:controller} extracts a safe stabilizing controller from a compatible CLF–CBF pair, and Theorem \ref{thm:clbf} (via Lemma \ref{lem:Tx}) strengthens the classical converse Lyapunov argument to enforce the required boundary condition.

\subsection{Extracting a safe stabilizing controller from CLF-CBF}

\begin{assumption}[Strict compatibility] \label{Assmp:SC}
The set\, $\mathcal{C}$ defined by (\ref{safe-set}) is compact and contains the origin in its interior. Let $\D$ be an open set containing $\C$ and $V:\,\D \rightarrow \mathbb{R}$ be a continuously differentiable function. Furthermore, suppose that the following strict compatibility assumption holds:\\
1) For every $x \in \mathcal{D} \setminus \{0\}$, there exists $u\in\Real^m$ for which 
    \begin{equation}\label{eq:CLFInt}
        L_f V(x) + L_g V(x) u < 0.
    \end{equation}
2) For every $x \in \partial \C$, there exists $u\in\Real^m$  for which 
    \begin{equation}\label{eq:CLFBoundary}
        L_f V(x) + L_g V(x) u < 0,
    \end{equation}
    \begin{equation}\label{eq:CBF}
        L_f h(x)+L_g h(x) u > 0.
    \end{equation}
\end{assumption}

\begin{rem}
We impose the strict form of inequality~\eqref{eq:CBF}, rather than the non-strict form~\eqref{eq:CBF-condition}, to ensure robust invariance of $\partial \mathcal{C}$. In particular, the vector field along $\partial \mathcal{C}$ must point strictly inward under a suitable control input so trajectories cannot remain tangent to the boundary. On the other hand, we relax the compatibility conditions in~\cite{mestres2022optimization,ong2019universal} to hold only on $\partial \C$, as this weaker requirement still guarantees controlled forward invariance (see Remark~\ref{rem:boundary}). These compatibility conditions preclude the impossibility results in \cite{braun2017existence,braun2020comment} for the existence of a CLBF; see the proof of Proposition \ref{prop:controller} below for the extraction of a $C^1$ safe stabilizing controller under these assumptions. Building on Proposition \ref{prop:controller}, Theorem \ref{thm:clbf} then establishes the existence of a converse CLBF with a prescribed boundary condition in the sense of Definition \ref{defn:CLBF}.
\end{rem}

To extract a controller with the desired regularity around the origin, we need one of the following assumptions.

\begin{assumption}[Stabilizable linearization]\label{as:stabilizability}
Let $f$ and $g$ be continuously differentiable, $A := \frac{\partial f}{\partial x}(0)$, and $B := g(0)$. We assume that $(A, B)$ is stabilizable. Furthermore, let $P=\nabla^2 V(0)$, where $V$ is from Assumption \ref{Assmp:SC}. There exists some matrix $K\in\Real^{m\times n}$ such that $\bar A:=A+BK$ is Hurwitz and $P\bar A+\bar A^\top P$ is negative definite. 
\end{assumption}

\begin{assumption}[Small control property]\label{as:small-contrl}
For each $\eps>0$, there exists some $\delta>0$ such that, if $x\neq 0$ and $\abs{x}<\delta$, then there exists some $u$ such that $\abs{u}<\eps$ and (\ref{eq:CLFInt}) holds. 
\end{assumption}

Under Assumption \ref{Assmp:SC} and either Assumption \ref{as:stabilizability} or \ref{as:small-contrl}, we can show that there exists a safe stabilizing controller with respect to the safe set $\C$ with desired properties. 

\begin{proposition}\label{prop:controller}
Let Assumption \ref{Assmp:SC} hold. 
\begin{enumerate}[(a)]
    \item If Assumption \ref{as:stabilizability} also holds, then there exists a $C^1$ controller $u:\,\mathbb{R}^n \to \mathbb{R}^m$ such that $x=0$ is exponentially stable and the set\, $\C$ is forward invariant for the closed-loop system (\ref{eq:clsys}) and contained in the domain of attraction of $x=0$. 
    \item If Assumption \ref{as:small-contrl} also holds, then there exists a  controller $u:\,\mathbb{R}^n \to \mathbb{R}^m$ that is continuous on $\mathbb{R}^n$ and $C^1$ on $\mathbb{R}^n\setminus\set{0}$ such that $x=0$ is asymptotically stable and the set\, $\C$ is forward invariant for the closed-loop system (\ref{eq:clsys}) and contained in the domain of attraction of $x=0$.
\end{enumerate}
\end{proposition}

\begin{proof}
    We first show part (a). Denote the CLF condition as 
    $\varphi(x, u) = L_f V(x) + L_g V(x) u$ 
    and the CBF condition as 
    $\psi(x, u) = - L_f h(x) - L_g h(x) u.$ 
    By Assumption \ref{as:stabilizability}, we have that (\ref{eq:sys}) is locally exponentially stabilizable by a linear controller $u_0(x) := Kx$, where $K \in \mathbb{R}^{m \times n}$, and $V$ is a local Lyapunov function for exponential stability of the closed-loop system. This implies that there exists a radius $r >0$ such that the controller $u_0(x)$ satisfies 
    $
    \varphi(x, u_0(x))\le -\mu V(x),
    $
    on $B_r(0)$ for some $\mu>0$. For some radius $0 < r_1 < r$, set $U_0 = \{ x \in \mathcal{C} : \|x\| < r \}$
    and $U_1 = \{ x \in \mathcal{D} : \|x\| > r_1 \}.$ Note that $U_0 \cup U_1 = \mathcal{D}$.
    
    Now we will construct a controller by a partition of unity argument on the set $U_1$. This essentially patches local feedback laws, each satisfying the CLF and CBF inequalities on small neighborhoods, into a single $C^1$ controller on $\mathcal D$.
    
    By Assumption \ref{Assmp:SC}, for all $x \in U_1$, there exists $u_x$ such that $\varphi(x, u_x) < 0$. In particular, because of (\ref{eq:CLFBoundary}) and (\ref{eq:CBF}), for $x \in\partial\mathcal{C}$, we have that 
    $\varphi(x, u_x) < 0$ and $\psi(x, u_x) < 0.$ 
     Set $\epsilon_{x} = \min \{-\varphi(x, u_x), - \psi(x, u_x) \} > 0$ for $x \in\partial\mathcal{C}$ and $\epsilon_{x} = -\varphi(x, u_x)$ for $x \in \mathcal{D} \setminus (\partial \mathcal{C} \cup \{ 0\})$. Since $\varphi$ and $\psi$ are continuous in the first argument, this implies that $\forall x \in U_1 \cap\partial\mathcal{C}^c$ there exists $\delta_{x} > 0$ such that, for all $y \in U_{x} := B_{\delta_{x}}(x) \subset U_1 \cap\partial\mathcal{C}^c$, we have that $\varphi(y, u_{x}) \leq -\epsilon_{x}/2$. Similarly, for $x \in\partial\mathcal{C}$, there exists $\delta_x > 0$ such that, for all $y \in U_{x} := B_{\delta_{x}}(x) \cap U_1$, we have that $\varphi(y, u_{x}) \leq -\epsilon_{x}/2$ and $\psi(y, u_{x}) \leq -\epsilon_{x}/2$.

    By construction, $\{U_{x}\}_{x}$ is an open cover of $U_1$. By a partition of unity (see, e.g., \cite[Theorem 2.23]{lee2003introduction}), we can find a locally finite collection of $C^1$ functions $\{\rho_\alpha:\,U_1 \to [0,1] \}_{\alpha \in U_1}$ for which $\operatorname{supp}(\rho_x) \subset U_{x}$ and $\sum_{\alpha \in U_1} \rho_\alpha(x) = 1$ for all $x \in U_1$. Form the function 
    $
    u_1(x) = \sum_{\alpha \in U_1} \rho_\alpha(x) u_{\alpha},
    $
    where $u_\alpha$ are the vectors used for setting $\epsilon_x$ as above. Clearly, $u_1$ is $C^1$. Since $\varphi$ is affine in $u$, it follows that 
    $
        \varphi(x, u_1(x)) = \sum_{\alpha \in U_1} \rho_{\alpha}(x) \varphi(x, u_{\alpha}).
    $
    For $x \in U_1$, there exists a finite set of indices $I(x) \subset U_1$ for which $\rho_\alpha(x) = 0$ for $\alpha \notin I(x)$ and $\varphi(x, u_\alpha) \le -\epsilon_\alpha/2$ for $\alpha \in I(x)$. Taking $\epsilon_1(x) := \min\{\epsilon_\alpha/2:\,{\alpha \in I(x)}\} > 0$, we obtain a function $\epsilon_1(x) > 0$ such that 
    $
    \varphi(x, u_1(x)) \leq -\epsilon_1(x) < 0.
    $  
    Furthermore, since by construction $x \in\partial\mathcal{C}$ can only be covered by neighborhoods $U_x$ for $x \in\partial\mathcal{C}$, by a similar argument, it follows that $\varphi(x, u_{1}(x)) \leq - \epsilon_1(x)$ and $\psi(x, u_{1}(x)) \leq - \epsilon_1(x)$ for all $x \in\partial\mathcal{C}$. 
    
    We now extend $u_1$ to a $C^1$ function defined on $\mathcal{D}$. Since $U_0 \cup U_1$ is an open cover of $\mathcal D$, by a partition of unity again, we can find two $C^1$ weight functions $\rho'_0, \rho'_1 : \mathcal{D} \to [0,1]$ with $\rho'_0 + \rho'_1 = 1$, $\operatorname{supp}(\rho'_0) \subset U_0, \operatorname{supp}(\rho'_1) \subset U_1$ and define the global controller $k: \mathcal{D} \to \mathbb{R}^m$ by 
    $
    k(x) = \rho'_0(x) u_0(x) + \rho'_1(x) u_1(x).
    $
    It follows that, on $B_{r_1}(0)\subset U_0\setminus U_1$, we have $k(x)=u_0(x)$ and $\varphi(x, k(x))\le -\mu V(x)$. Similar to the argument above, setting $\epsilon_2 := \min_{x\in \overline{U_0\cap U_1}}(\mu V(x),\epsilon_1(x))>0$, we have $\varphi(x, k(x))\le -\epsilon_2$ on $U_0\cap U_1$. On $U_1\setminus U_0$, we have $k(x)=u_1(x)$ and $\varphi(x, k(x))\le \epsilon_1(x)<0$ as shown.

    Finally, we can extend easily $k(x)$ from a neighborhood of $\mathcal C$ to $\mathbb{R}^n$, while preserving the $C^1$ regularity of $k(x)$. This finishes the proof of part (a).

    To prove part (b), note that, by \cite{sontag1989universal}, 
    the small control property guarantees the existence of a controller $\hat u_0(x)$ that is continuous at the origin and $C^1$ on $\mathcal{D} \setminus \{0\}$. Moreover, the controller is guaranteed to satisfy (\ref{eq:CLFInt}). By fixing radii $0 < r_1 < r$ such that $B_r(0) \subset \operatorname{Int}(\mathcal{C})$, a similar partition of unity argument results in the controller
    $
    k(x) = \rho_0(x) \hat u_0(x) + \rho_1(x) u_1(x),
    $ 
    where $\operatorname{supp}(\rho_0(x)) \subset U_0$, $\operatorname{supp}(\rho_1(x)) \subset U_1$, $\hat u_0(x)$ satisfies the CLF condition on $B_r(0)$ and $u_1(x)$ satisfies the CLF condition on $U_1$ and jointly satisfies the CLF and CBF conditions on $\delta \mathcal{C}$. It is then straightforward to show that $k(x)$ satisfies (\ref{eq:CLFInt}) on $\mathcal{D} \setminus \{0\}$ and conditions (\ref{eq:CLFBoundary}, \ref{eq:CBF}) on $\partial \mathcal{C}$. Therefore, the control $u_0(x)$ is continuously differentiable on $\mathcal{D} \setminus \{0\}$ and continuous at the origin. We can similarly extend $k(x)$ from a neighborhood of $\mathcal C$ to $\mathbb{R}^n$. 
\end{proof}

\begin{rem}\label{rem:controller}
We highlight the differences between Proposition~\ref{prop:controller} and similar results in the literature, such as Corollary~4.2.2 of~\cite{Ong2022uniting}, which extends Artstein's theorem~\cite{artstein1983stabilization} (see also \cite{sontag1989universal}). First, safe exponential stabilization with a controller that is $C^1$ at the origin as formulated in part (a) Proposition~\ref{prop:controller} has not been presented in the context combining CLF and CBF. Second, the CBF-CLF comparability condition is required only on the boundary of the safe set, while the controller $k(x)$ is constructed globally on the open set $\mathcal{D}$. This is weaker than the compatibility condition considered in, e.g., \cite{li2023graphical,dai2024verification},  and has potential to reduce computational conservatism; see \cite{liu2025computing}.
\end{rem}

\subsection{Construction of a $C^1$ CLBF}

We rely on the following regularity property of the hitting time of the boundary of the safe set $\C$ for solutions of the closed-loop system. The proof is provided in the Appendix.
\begin{lemma}\label{lem:Tx}
Let the assumptions of Proposition \ref{prop:controller} hold. Let $\D$ denote the domain of attraction of $x=0$ for the closed-loop system (\ref{eq:clsys}) under the controller constructed by Proposition \ref{prop:controller}. For each $x\in \D\setminus\set{0}$, there exists a unique $T(x)\in\Real$ such that $\phi(T(x),x)\in \partial \C$, i.e., $h(\phi(T(x),x))=0$, and $T(x)$ is continuously differentiable in $x$ on $\D\setminus\set{0}$. 
Furthermore, if Assumption \ref{as:stabilizability} holds, then $\abs{\nabla T(x)}=\mathcal{O}(1/\abs{x})$ as $x\ra 0$. 
\end{lemma}

The role of the hitting time \(T(x)\) is to enable trajectory-wise normalization of the converse Lyapunov function so that it satisfies the prescribed boundary condition on \(\partial \C\).

\begin{thm}\label{thm:clbf}
Suppose that Assumption \ref{Assmp:SC} holds, and either Assumption \ref{as:small-contrl} or \ref{as:stabilizability} holds. Then there exists an open set $\D\supset \C$ and a continuously differentiable function $W:\,\D\ra\Real$ such that $W$ satisfies all the conditions in Definition \ref{defn:CLBF}, i.e., $W$ is a CLBF on $\D$ for (\ref{eq:sys}) w.r.t. to the safe set $\C$.
\end{thm}

\begin{proof}
Consider the vector field $F$ of the closed‐loop system (\ref{eq:clsys}) under the controller constructed in Proposition \ref{prop:controller}, and let $\D$ denote the domain of attraction of the origin. By classical results on converse Lyapunov theorems; see, e.g., \cite{lin1996smooth,teel2000smooth,kurzweil1963inversion}, there exists a smooth (i.e., $C^\infty$) Lyapunov function $V:\,\D\ra\Real$ for (\ref{eq:clsys}). 

We first assume that Assumption \ref{as:stabilizability} holds. Define $\omega(x)=-\nabla V(x)^{\top}  F(x)$ for $x\in \D$. Then $\omega$ is positive definite on $\D$. It can be shown that (see, e.g., \cite[Proposition 2]{liu2025physics})
$
V(x) = \int_0^\infty \omega(\phi(t,x))dt, 
$
for all $x\in \D$. For $x\in \D\setminus\set{0}$, define 
$
\omega_1(x) = \frac{\omega(x)}{V(\phi(T(x),x))}, 
$
where $T(x)$ is from Lemma \ref{lem:Tx}. 

Since $V$ is $C^\infty$ and $F$ is $C^1$, we have that $\nabla\omega$ is $C^1$ on $\D$. Moreover, $\nabla\omega(0)=0$, since $\nabla V(0)=0$ (because $0$ is a local minimum of $V$) and $F(0)=0$. Hence $\omega(x)=o(\lvert x\rvert)$ as $x\to 0$. Because $\phi(T(x),x)\in\partial\C$ and the denominator remains positive and bounded away from zero, $\omega_1$ is continuous and positive on $\D\setminus\{0\}$, with $\omega_1(x)=o(\lvert x\rvert)$ as $x\to 0$. Defining $\omega_1(0)=0$ makes $\omega_1$ continuous on $\D$ and differentiable at $0$, with $\nabla\omega_1(0)=0$.

We further show that $\omega_1$ is $C^1$ on $\D$. We have
\begin{align*}
&\nabla \omega_1(x) \\
&\quad = \frac{ \nabla \omega(x) \cdot V(\phi(T(x),x)) - \omega(x) \cdot \nabla \left[ V(\phi(T(x),x)) \right] }{ V(\phi(T(x),x))^2 },
\end{align*}
where
\begin{align*}
& \nabla \left[ V(\phi(T(x), x)) \right]  = \nabla V(\phi(T(x), x))^\top F(\phi(T(x), x)) \nabla T(x) \\
&\quad + \frac{\partial \phi}{\partial x}(T(x), x)^\top\nabla V(\phi(T(x), x)). 
\end{align*}
It follows that $\omega_1(x)$ is continuous on $\D\setminus\set{0}$. As shown in the proof of Lemma \ref{lem:Tx}, $\norm{\frac{\partial \phi}{\partial x}(T(x), x)}=\mathcal{O}(1/\abs{x})$ and $T(x)=\mathcal{O}(1/\abs{x})$ as $x\ra 0$. Furthermore, since $\phi(T(x),x)\in\partial \C$, $\nabla w(x)\ra 0$, and $\omega(x)=o(\abs{x})$ as $x\ra 0$, we conclude from the calculation of $\nabla \omega_1(x)$ above that $\nabla \omega_1(x)\ra 0$ as $x\ra 0$. This verifies that $\omega_1$ is continuously differentiable on $\D$. 

Define 
$
W(x) = \int_0^\infty \omega_1(\phi(t,x)) dt.  
$
It follows from \cite[Proposition 2]{liu2025physics} that $W$ is continuously differentiable on $\D$ and solves the PDE
\begin{equation}
    \label{eq:lyap_pde}
    \nabla W(x)\cdot F(x) = -\omega_1(x),\quad x\in \D, 
\end{equation}
where $F(x)=f(x)+g(x)u(x)$ and $u$ is the controller constructed in Proposition \ref{prop:controller}. Clearly, this implies that condition (1) of Definition \ref{defn:CLBF} holds. 

Note that $x\mapsto V(\phi(T(x),x))$ remains constant on the trajectory $\phi(t,x)$. We have
\begin{equation}\label{eq:W}
W(x) = \frac{V(x)}{V(\phi(T(x),x))}.
\end{equation}
For $x\in \D\setminus \C$, we have $T(x)>0$ and 
$$
V(\phi(T(x),x)) = V(x) - \int_0^{T(x)}\omega(\phi(t,x))dt<V(x),
$$
which gives $W(x)>1$. For $x\in \operatorname{Int}(\mathcal{C})$, $T(x)<0$ and 
$$
V(\phi(T(x),x)) = V(x) - \int_0^{T(x)}\omega(\phi(t,x))dt > V(x),
$$
which gives $W(x)<1$.

For $x\in \partial \C$, we have $T(x)=0$ and $V(\phi(T(x),x))=V(x)$, which gives $W(x)=1$. 

Combining these gives conditions 2)--3) of Definition~\ref{defn:CLBF}. 

Now we consider the case under Assumption~\ref{as:small-contrl}. We can still define \(W\) using \eqref{eq:W}. By Lemma~\ref{lem:Tx}, the smoothness of \(V\), and the continuous differentiability of \(\phi(t,x)\) with respect to \(x\) away from the origin, \(W\) is continuously differentiable on \(\D\setminus\{0\}\) and continuous on $\D$ (in fact, we can show that it is differentiable on $\D$). We can check that \(W\) satisfies all the conditions in Definition~\ref{defn:CLBF}, except that it may fail to be continuously differentiable at \(x=0\). We then apply a smoothening function \(\rho\) so that \(W_1:=\rho\circ W\) is $C^1$ on $\mathcal D$ and meets all the requirements of Definition~\ref{defn:CLBF} (see \cite[Lemma 17]{teel2000smooth} and also \cite[Theorem 4.3]{lin1996smooth}). In fact, such a $\rho$ can be chosen to be of class $\K_\infty$\footnote{We say that a continuous function $\alpha:[0, \infty)\ra [0,\infty)$ belongs to class $\mathcal{K}_\infty$, if it is strictly increasing,  $\alpha(0)=0$, and $\alpha(s)\ra \infty$ as $s\ra \infty$.} and satisfy $\rho(s) \le s \rho'(s)$ for all $s \ge 0$. Without loss of generality, we may also assume that \(\rho(1) = 1\); otherwise, we can replace $\rho$ with $\rho/\rho(1)$ to satisfy this property. Together with (\ref{eq:lyap_pde}), this implies that $\nabla W_1\cdot F(x)=\rho'(W_1)\nabla W\cdot F=-\rho'(W_1)\omega_1(x)<0$ for all $x\in\mathcal{D}\setminus\set{0}$. Furthermore, because $\rho$ is strictly increasing and $\rho(1)=1$, we have 
$\mathcal C=\set{x\in\mathcal D:\,\rho\circ W(x)\le 1}$ and $\partial \mathcal C=\set{x\in\mathcal D:\,\rho\circ W(x)= 1}$. This shows that $W_1$ satisfies all the conditions of Definition~\ref{defn:CLBF}.
\end{proof}

The following example shows that, without Assumption~\ref{as:stabilizability}, the function $W$ defined by~(\ref{eq:W}) may fail to be $C^1$ at $x = 0$. In such cases, an additional smoothening step is indeed necessary to obtain a $C^1$ CLBF on $\mathcal{D}$.

\begin{example}
    Consider the two-dimensional system in polar coordinates: 
    $
    \dot r = - r^3,
    $
    $
    \dot\theta = \frac{1}{\sqrt{r}},
    $
    where $r>0$; the case $r=0$ corresponds to the origin, which is an equilibrium point.  
    It can be easily checked that the vector field is continuous on $\Real^2$ and smooth on $\Real^2\setminus\set{0}$. This regularity aligns with that guaranteed for the closed-loop system under the controller obtained by Proposition \ref{prop:controller}.
    Consider $h=1-r^2$ and $V=4r^2+r^5\sin\theta$. Then $V$ is a valid Lyapunov function on an open set containing $\,\C=\set{r\le 1}$. Indeed, we have
    $
    \dot V = -8 r^4 - 5 r^7 \sin\theta + r^{9/2} \cos\theta \le -8 r^4 + 5 r^7 + r^{9/2},
    $
    which is negative definite in an open neighborhood of $\mathcal C$. 
    The hitting time $T(x)$ in Lemma \ref{lem:Tx} can be directly computed as 
    $
    T(x) = (1-1/r^2)/2. 
    $
    We can then explicitly solve 
    $$
    W(x)=\frac{V(x)}{V(\phi(T(x),x))} = \frac{4r^2+r^5\sin(\theta)}{4+\sin(\theta + \frac25(1-\frac{1}{r^{\frac52}}))}.
    $$
    It can be checked by definition that $W(x)$ is differentiable on $\Real^2$, continuously differentiable on $\Real^2\setminus\set{0}$, but fails to be continuously differentiable at $x=0$. 
\end{example}

The example below shows that a standard converse Lyapunov theorem may not yield a Lyapunov function whose sublevel set exactly matches $\mathcal{C}$. Such exact representation can offer computational advantages when seeking to maximize the safe stabilization region within a safe domain \cite{liu2025computing}.
\begin{example}
    Consider the one-dimensional system $x'=-x$. The origin is globally asymptotically stable and the domain of attraction is $\mathbb{R}$. A standard converse Lyapunov theorem (e.g., \cite{hafstein2004constructive}) gives a global Lyapunov function 
$
V(x) = \int_0^\infty |\phi(t,x)|^2dt = \frac{x^2}{2}. 
$
Consider a safe set defined by $\mathcal{C} = [-2,1] = \set{x:\,h(x)\ge 0}$ with 
$h(x) = (x+2)(1-x)$. Since there is no control input, strict compatibility of $h$ and $V$ trivially holds because $\mathcal C$ is forward invariant and contained in the domain of attraction. Nonetheless, no sublevel sets of $V$ exactly capture $\mathcal{C}$. By direct computation, one can verify that the construction in the proof of Theorem \ref{thm:clbf} yields $W(x)=x^2$ for $x\ge 0$ and $W(x)=\frac14 x^2$ for $x<0$. Note that $W$ is $C^1$ and satisfy the conditions of Definition \ref{defn:CLBF}. 
\end{example}

\section{Conclusion}

In this work, we established a converse theorem on the existence of CLBFs in the context of safe stabilization for control systems. Specifically, we showed that, under mild assumptions, the existence of a strictly compatible pair of control Lyapunov and control barrier functions is equivalent to the existence of a single continuously differentiable Lyapunov function that simultaneously certifies asymptotic stability and safety. In proving this result, we also established regularity properties and asymptotic estimates on the backward hitting time of a compact set, assuming the set contains an asymptotically stabilizing equilibrium point and its boundary is robustly invariant. It would be of interest in future work to explore how the weakened notion of CLF–CBF compatibility and the converse theorem can be exploited to compute less conservative CLBFs (preliminary results reported in \cite{liu2025computing}). It would also be interesting future work to relax the strict compatibility on $\partial\C$, remove compactness of $\C$, and handle time-varying safety constraints.

\bibliographystyle{plain}        %
\bibliography{acc26} 

\appendices

\section{Proof of Lemma \ref{lem:Tx}}

\begin{proof}
Recall that $\C\subset\D$, $\C$ is compact, and $0\in \operatorname{Int}(\mathcal{C})$. By the fact that $\C$ is robustly invariant, i.e., $\nabla h(x)^{\top} F(x)>0$ for all $x\in \partial C$, solutions can only cross $\partial \C$ once. It follows that for any $x\in \D\setminus\set{0}$, if $T(x)$ exists, it must be unique.

For each $x\in \D\setminus \C$, by the definition of $\C$, we have $h(x)<0$. We also know $h(0)>0$. Let $\eps>0$ be such that $h(x)>0$ for all $x\in B_\eps(0)$. By asymptotic stability of $x=0$, there exists $T'>0$ such that $\abs{\phi(t,x)}<\eps$ for all $t\ge T'$. We have $h(\phi(0,x))<0$ and $h(\phi(T',x))>0$. It follows by the intermediate value theorem that there exists $T(x)\in (0,T')$ such that $h(\phi(T(x),x))=0$. 

Now consider $x\in \C\setminus\set{0}$. Since $\C$ is compact and contained within the domain of attraction, for each $\eps>0$, there exists some $T_{\mathcal C}(\eps)>0$ such that for all $x\in \C$, we have $\abs{\phi(t,x)}<\eps$ when $t\ge T_{\mathcal C}(\eps)$. For each $x\in \C\setminus\set{0}$, pick $\eps\in (0,\abs{x})$. Suppose that $\inf\set{t>0:\,h(\phi(-t,x))=0}=\infty.$  
It follows that $\phi(t,x)\in \C$ for all $t\le 0$. This contradicts $\abs{\phi(t,\phi(-t,x))}=\abs{x}>\eps$ for all $t\ge 0$. Hence, $\inf\set{t>0:\,h(\phi(-t,x))=0}<\infty$. By continuity, we  can let $T(x)=-\inf\set{t>0:\,h(\phi(-t,x))=0}$ and we have that $\phi(T(x),x)\in \partial \C$. 

To show $T(x)$ is continuously differentiable in $x$ on $\D\setminus\set{0}$, we use the implicit function theorem. Consider
$
\psi(t,x): = h(\phi(t,x)). 
$
Clearly, $\psi$ is continuously differentiable in $t$ and $x$. We proved that for each $x\in \D\setminus\set{0}$, there exists $T(x)$ such that $\psi(T(x),x)=0$. The partial derivative of $\psi$ w.r.t. $t$ at $(T(x),x)$ is
$$
\frac{\partial \psi}{\partial t}(T(x),x) = \nabla h(\phi(T(x),x))^{\top} F(\phi(T(x),x)) > 0,
$$
as $\phi(T(x),x))\in \partial \C$. By the implicit function theorem, $T(x)$ is continuously differentiable on $\D\setminus\set{0}$. 

Under Assumption \ref{as:stabilizability}, it follows that the closed-loop system is exponentially stable by Proposition \ref{prop:controller}. To estimate the growth of $\nabla T(x)$ as $x\ra 0$, differentiate the identity 
$
h(\phi(T(x),x))=0.
$
We have
\begin{equation}\label{eq:dT}
\nabla T(x) = -\frac{\frac{\partial \phi}{\partial x}(T(x), x)^{\top}  \nabla h(\phi(T(x), x)) }{ \nabla h(\phi(T(x), x))^{\top} F(\phi(T(x), x)) }.
\end{equation}
Note that $\phi(T(x),x))\in \partial \C$ and hence $\nabla h(\phi(T(x), x))$ and $1/\nabla h(\phi(T(x), x))^{\top} F(\phi(T(x), x))$ are both bounded on $\D\setminus\set{0}$. To bound $\nabla T(x)$, it remains to bound $\frac{\partial \phi}{\partial x}(T(x), x)$. 

To estimate $\frac{\partial \phi}{\partial x}(T(x), x)$ as $x\ra 0$, define $\Phi$ to be the fundamental matrix solution of the initial value problem \cite[p.~82]{perko2001differential} 
$$
\dot{\Phi}(t) = A(t)\Phi(t), \quad \Phi(0) = I,
$$
where
$ 
A(t) = \frac{\partial F}{\partial x}(\phi(t, x)). 
$ 
Then
\begin{equation}\label{eq:Phi}
\frac{\partial \phi}{\partial x}(T(x), x) = \Phi(T(x)).
\end{equation}
Consider any $x\in \operatorname{Int}(\mathcal{C})$. On $[T(x),0]$, $\phi(t,x)\in \C$. It follows that $\norm{A(t)}\le \Lambda$ on $[T(x),0]$ for some $\Lambda>0$ independent of the choice of $x\in \operatorname{Int}(\mathcal{C})$. As a result, by Gronwall's inequality (see, e.g., \cite[p.~79]{perko2001differential}), each column $\Phi_i(t)$ of $\Phi(t)$ satisfies 
\begin{equation}\label{eq:Phi_i}
\norm{\Phi_i(t)}\le e^{\Lambda \abs{t}}    
\end{equation}
on $[T(x),0]$. 

We now estimate $T(x)$ as $x\ra 0$. We use a linearization argument. Let $\hat A= \frac{\partial F}{\partial x}(0)$, $g(x)=F(x)-Ax$, and $P$ be the unique positive definite solution to $PA+A^{\top} P=-I$. Let $V_P(x)=x^{\top} Px$. Pick $c>0$ sufficiently small such that $V_P(x)\le c$ implies $x\in \operatorname{Int}(\mathcal{C})$ and $2x^{\top} Pg(x)\le -\frac12\abs{x}^2$. This is always possible because $g(x)/\abs{x}\ra 0$ as $x\ra 0$. Then we have 
\begin{align*}
\nabla V_P(x)^{\top}  F(x) & = x^{\top} (PA+A^{\top} P)x + 2x^{\top} Pg(x) \\
& = - \abs{x}^{\top}  + 2x^{\top} Pg(x) \le - \frac12 \abs{x}^2   \\
& \le - \frac{1}{2\lambda_{\max}(P)}V_P(x)
\end{align*}
for all $x$ such that $V_P(x)\le c$. It follows that the set $\set{x:\,V_P(x)\le c}\subset \operatorname{Int}(\mathcal{C})$ is robustly positive invariant. Let
\begin{equation}\label{eq:hatT}
\hat{T} := \min_{x : x^\top P x = c} T(x).     
\end{equation}
Now consider any $x$ such $V_P(x)<c$. Similar to the argument we used to show existence of $T(x)$, we can show that there exists $T_P(x)<0$ such that $V_P(\phi(T_P(x),x))=c$. The semigroup property of the flow implies $x=\phi(-T_P(x),\phi(T_P(x),x))$. By the quadratic Lyapunov analysis above, we have 
\begin{align*}
\lambda_{\min}(P)\abs{x}^2 &\le V_P(x) \le V_P(\phi(T_P(x),x)) e^{\frac{T_P(x)}{2\lambda_{\max}(P)}}\\
&= ce^{\frac{T_P(x)}{2\lambda_{\max}(P)}},
\end{align*}
which implies 
$$
T_P(x)\ge 2\lambda_{\max}(P)\log \frac{\lambda_{\min}(P)}{c} + 4\lambda_{\max}(P)\log\abs{x}.
$$
By the semigroup property of the flow again, we have 
$$
T(x) = T_P(x) + T(\phi(T_P(x),x)) \ge T_P(x) + \hat{T},
$$
where $\hat{T}$ is defined in (\ref{eq:hatT}). It follows that $\abs{T(x)}\le c_1 +c_2\log\frac{1}{\abs{x}}$ for some $c_1>0$ and $c_2>0$. Putting this back to (\ref{eq:Phi_i}) and (\ref{eq:Phi}), we obtain $\norm{\frac{\partial \phi}{\partial x}(T(x), x)}\le c_3/\abs{x}$ for some $c_3>0$. Finally, substituting this bound back to (\ref{eq:dT}), we obtain that $\abs{\nabla T(x)}\le c_4/\abs{x}$ for some $c_4>0$ and all $x$ such that $V_P(x)<c$. This shows that $\abs{\nabla T(x)}=\mathcal{O}(1/\abs{x})$ as $x\ra 0$ and completes the proof.
\end{proof}

\end{document}